\documentclass[12pt]{amsart}
\usepackage{fullpage}
\usepackage{ulem}
\usepackage{amsmath}
\usepackage{amsthm}
\usepackage{amsbsy}
\usepackage{amsrefs}
\usepackage{amssymb}
\usepackage{hyperref}
\usepackage{amsfonts}

\usepackage{enumitem}
\usepackage{amsfonts}
\numberwithin{equation}{section}
\newtheorem{theorem}{Theorem}[section]
\newtheorem{lemma}[theorem]{Lemma}

\newtheorem{remark}[theorem]{Remark}

\newcommand{\C}{\mathbb{C}}
\newcommand{\R}{\mathbb{R}}
\newcommand{\Q}{\mathbb{Q}}

\newcommand{\add}[2]{\displaystyle\sum_{#1}^{#2}}
\newcommand{\p}[1]{\left( #1 \right)}

\newcommand{\Lsym}{L(\textup{Sym}^n f,s)}
\newcommand{\new}{S_k^{\textup{new}}(\Gamma_0(N))}

\newcommand{\sym}{\textup{Sym}^n f}

\title[The Error Term in the Sato-Tate Conjecture]{The Error Term in the Sato-Tate Conjecture}
\author{Jesse Thorner}
\address{Jesse Thorner, Department of Mathematics and Computer Science, Emory University, Atlanta, Georgia 30322}
\email{jesse.thorner@gmail.com}
\date{\today}
\begin{document}
\maketitle
\vspace{-.3in}
\begin{abstract}
Let $f(z)=\sum_{n=1}^\infty a(n)e^{2\pi i nz}\in\new$ be a newform of even weight $k\geq2$ that does not have complex multiplication.  Then $a(n)\in\R$ for all $n$, so for any prime $p$, there exists $\theta_p\in[0,\pi]$ such that $a(p)=2p^{(k-1)/2}\cos(\theta_p)$.  Let $\pi(x)=\#\{p\leq x\}$.  For a given subinterval $I\subset[0,\pi]$, the now-proven Sato-Tate Conjecture tells us that as $x\to\infty$,
\[
\#\{p\leq x:\theta_p\in I\}\sim \mu_{ST}(I)\pi(x),\quad \mu_{ST}(I)=\int_{I} \frac{2}{\pi}\sin^2(\theta)~d\theta.
\]
Let $\epsilon>0$.  Assuming that the symmetric power $L$-functions of $f$ are automorphic and satisfy Langlands functoriality, we prove that as $x\to\infty$,
\[
\#\{p\leq x:\theta_p\in I\}=\mu_{ST}(I)\pi(x)+O\left(\frac{x}{(\log x)^{9/8-\epsilon}}\right),
\]
where the implied constant is effectively computable and depends only on $k,N,$ and $\epsilon$.
\end{abstract}

%%%%%%%%%%%%%%%%%%%%%%%%%%%%%%%%%%%%%%%%%%%

\section{Introduction and Statement of Results}

%%%%%%%%%%%%%%%%%%%%%%%%%%%%%%%%%%%%%%%%%%%

Let
\begin{equation}
f(z)=\sum_{n=1}^\infty a(n)q^n\in S_{k}^{\textup{new}}(\Gamma_0(N)),\quad q=e^{2\pi iz}
\end{equation}
be a newform of even weight $k\geq2$ with trivial character.  Then $a(n)\in\R$ for all $n$, and as a consequence of Deligne's proof of the Weil conjectures, for each prime $p$ there exists an angle $\theta_p\in[0,\pi]$ such that
\begin{equation}
a(p)=2p^{(k-1)/2}\cos(\theta_p).
\end{equation}
For a newform associated to an elliptic curve $E/\Q$ (in which case $k=2$), Sato and Tate independently conjectured the distribution of the sequence $\{\theta_p\}$ as $p$ varies through the primes; the following generalization of the conjecture for $k\geq2$ was proven by Barnett-Lamb, Geraghty, Harris, and Taylor \cite{sato-tate}.
\begin{theorem}[The Sato-Tate Conjecture]
Suppose that $f\in S_{k}^{\textup{new}}(\Gamma_0(N),\chi)$ does not have complex multiplication, and let $F:[0,\pi]\to\C$ be a Riemann-integrable function.  Then
\[
\lim_{x\to\infty}\frac{1}{\pi(x)}\sum_{p\leq x}F(\theta_p)=\int_0^\pi F(\theta)~d\mu_{ST},
\]
where $d\mu_{ST}$ is the Sato-Tate measure $\frac{2}{\pi}\sin^2(\theta)d\theta$.
\end{theorem}

%%%%%%%%%%%%%%%%%%%%%%%%%%%%%%%%%%%%%%%%%%%

Since Riemann-integrable functions can be uniformly approximated by step functions, if suffices for us to consider the function
\begin{equation}
\pi_{f,I}(x):=\#\{p\leq x:\theta_p\in I\},
\end{equation}
in which case Theorem \ref{sato-tate} tells us that if $I=[\alpha,\beta]\subset[0,\pi]$ is fixed, then
\begin{equation}
\label{sato-tate}
\pi_{f,I}(x)\sim \mu_{ST}(I)\pi(x).
\end{equation}

The Sato-Tate Conjecture governs much of the statistical behavior of the Fourier coefficients of $f$.  It is known \cite{Shahidi} that the Sato-Tate Conjecture follows from the analytic properties of the symmetric power $L$-functions associated to $f$ that are predicted by Langlands functoriality.  In order to bound the error in (\ref{sato-tate}), one must assume that all symmetric power $L$-functions of $f$ have these conjectured analytic properties.

%%%%%%%%%%%%%%%%%%%%%%%%%%%%%%%%%%%%%%%%%%%

There have been a number of estimates for the error in (\ref{sato-tate}) under the additional assumption that the symmetric power $L$-functions of $f$ satisfy the Generalized Riemann Hypothesis for symmetric power $L$-functions (GRH).  Under this additional assumption, building on the work of Murty \cite{Murty}, Bucur and Kedlaya \cite{BK} proved that if $f$ is the newform associated to an elliptic curve $E/\Q$ without complex multiplication, then
\[
\pi_{f,I}(x)=\mu_{ST}(I)\pi(x)+O(x^{3/4}\sqrt{\log(Nx)}),
\]
where $N$ is the conductor of $E$.  When $f\in S_k^{\textup{new}}(\Gamma_0(N))$ is a newform of even weight $k\geq2$ with squarefree level $N$ (such a newform necessarily does not have complex multiplication), Rouse and the author \cite{RT} proved a completely explicit version of the Sato-Tate Conjecture with a slight improvement in Murty's error term; this can be briefly stated as
\[
\pi_{f,I}(x)=\mu_{ST}(I)\pi(x)+O\left(\frac{x^{3/4}\log(Nkx)}{\log(x)}\right).
\]

It is important to understand the error term in the Sato-Tate Conjecture without the assumption of GRH.  The goal of this note is to prove the following result, providing such an error term.

\begin{theorem}
\label{main-theorem}
Let $f\in S_k^{\textup{new}}(\Gamma_0(N))$ be a newform of even weight $k\geq2$ and trivial character without complex multiplication.  Suppose that all of the symmetric power $L$-functions of $f$ are automorphic and satisfy Langlands functoriality.  If $[\alpha,\beta]\subset[0,\pi]$ is fixed, then for any $\epsilon>0$,
\[
\pi_{f,I}(x)=\mu_{ST}(I)\pi(x)+O\left(\frac{x}{(\log x)^{9/8-\epsilon}}\right),
\]
where the implied constant is effectively computable and depends only on $k,N,$ and $\epsilon$.
\end{theorem}

%%%%%%%%%%%%%%%%%%%%%%%%%%%%%%%%%%%%%%%%%%%

\section{Symmetric Power $L$-Functions}

%%%%%%%%%%%%%%%%%%%%%%%%%%%%%%%%%%%%%%%%%%%

We will adopt the notation $F\ll_{a}G$, or equivalently $F=O_a(G)$, to indicate that $\limsup_{x\to\infty}|F(x)/G(x)|<\infty$, where the limit superior may depend on $a$.  If there is no subscript for $\ll$ or $O(\cdot)$, then the implied constant is absolute.  We take $F\sim G$ to mean that $\lim_{x\to\infty}F(x)/G(x)=1$.

In this section we discuss the relevant background on the symmetric power $L$-functions of $f$.  First, we discuss the assumption that the symmetric power $L$-functions of $f$ are automorphic.  We then estimate the analytic conductor of $\Lsym$, a quantity that will be useful in determining dependence of important quantities on the level $N$, the weight $k$, and the symmetric power $n$.

\subsection{Automorphy and functoriality.}

Let $f(z)=\sum_{n=1}^\infty a(n)q^n\in\new$ be a newform of even weight $k\geq2$ without complex multiplication.  For each prime $p$, define $\theta_p\in[0,\pi]$ to be the angle for which $a(p)=2p^{(k-1)/2}\cos(\theta_p)$.  The newform $f$ has an associated $L$-function
\begin{equation}
L(f,s)=\sum_{n=1}^\infty \frac{a(n)}{n^{s+(k-1)/2}}=\prod_{p}\prod_{j=0}^1 (1-\alpha_p^j\beta_p^{1-j}p^{-s})^{-1}.
\end{equation}
It is known that $L(f,s)$ can be analytically continued to an entire function that satisfies a functional equation.  By Deligne's proof of the Weil conjectures, we know that $|\alpha_p|=|\beta_p|=1$ when $p\nmid N$ and $|\alpha_p|,|\beta_p|\leq1$ when $p\mid N$.  Because $f$ has trivial character, we have $\alpha_p=e^{i\theta_p}$ and $\beta_p=e^{-i\theta_p}$ for all primes $p\nmid N$.

%%%%%%%%%%%%%%%%%%%%%%%%%%%%%%%%%%%%%%%%%%%

For each $n\geq0$, the $n$-th symmetric power $L$-function of $f$ is the degree $n+1$ $L$-function given by the Euler product
\begin{equation}
\Lsym =\prod_{p}\prod_{j=0}^n (1-\alpha_p^j\beta_p^{n-j}p^{-s})^{-1}.
\end{equation}
When $n=0$, $\Lsym$ reduces to the Riemann zeta function $\zeta(s)$; when $n=1$, we obtain $L(f,s)$.  Conjecturally, there exists a functoriality lifting map on global automorphic functions that commutes with the local Langlands correspondence.  This would imply that for all $n\geq1$, $\Lsym$ is an automorphic $L$-function.  As a result, $\Lsym$ would have an analytic continuation to an entire function on $\C$, and this analytic continuation would satisfy a functional equation of the usual type.  Specifically, there would exist a positive integer $q_{\sym}$ (the {\it conductor}), a complex number $\epsilon_{\sym}$ of modulus 1 (the {\it root number}), and a function $\gamma(\sym,s)$ (the {\it gamma factor}) so that the function
\begin{equation}
\Lambda(\sym,s)=q_{\sym}^{s/2}\gamma(\sym,s)L(\sym,s)
\end{equation}
is an entire function of order 1 and satisfies the functional equation
\begin{equation}
\Lambda(\sym,s)=\epsilon_{\sym}\Lambda(\sym,1-s).
\end{equation}

Let $\Gamma(s)$ be the usual Gamma function, and let
\[
\Gamma_{\R}(s)=\pi^{-s/2}\Gamma(s/2),\qquad\Gamma_{\C}(s)=\Gamma_{\R}(s)\Gamma_{\R}(s+1).
\]
It is known \cite{CM, MS} that under our working assumptions, we have
\begin{equation}
\label{gamma-factor}
\gamma(\sym,s)=
\begin{cases}
\prod_{j=0}^{(n-1)/2}\Gamma_{\C}(s+(j+1/2)( k -1))  &\mbox{if $n$ is odd,} \\
\Gamma_{\R}\p{s+r}\prod_{j=1}^{n/2}\Gamma_{\C}(s+j( k -1)) & \mbox{if $n$ is even,}
\end{cases} 
\end{equation}
where $r=1$ if $n/2$ is odd and $r=0$ if $n/2$ is even.  Using the definitions of $\Gamma_{\R}(s)$ and $\Gamma_{\C}(s)$, we can express the $\gamma(\sym,s)$ as a constant multiple of
\begin{equation}
\pi^{-(n+1)s/2}\prod_{j=1}^{n+1}\Gamma\left(\frac{s+\kappa_{j,\sym}}{2}\right)
\end{equation}
for some appropriate numbers $\kappa_{j,\sym}\in\C$ with $1\leq j\leq n+1$.  The numbers $\kappa_{j,\sym}$ satisfy the inequality $|\kappa_{j,\sym}|\leq(n+1)\max_j|\kappa_{j,\textup{Sym}^1 f}|$.  For the rest of the paper, we will assume that $\Lsym$ is automorphic for all $n\geq1$, though this hypothesis is only known to be true unconditionally for $n=1,2,3,4$ by the work of Gelbart, Jacquet, Kim, and Shahidi \cite{GJ,Kim,KS1,KS2}.

%%%%%%%%%%%%%%%%%%%%%%%%%%%%%%%%%%%%%%%%%%%

Define the numbers $\Lambda_{\sym}(j)$ by
\begin{equation}
-\frac{L'}{L}(\sym,s)=\add{j=1}{\infty}\frac{\Lambda_{\sym}(j)}{j^s}.
\end{equation}
A straightforward computation shows that
\[
-\frac{L'}{L}(\sym,s)=\add{p}{}\add{m=1}{\infty}\p{\add{j=0}{n}(\alpha_p^{j} \beta_p^{n-j})^m}\log(p)p^{-ms}.
\]
Since $|\alpha_p|,|\beta_p|\leq 1$ for all primes $p$ (including the ramified ones, under our assumption of functoriality), it follows that for any positive integer $j$, we have
\begin{equation}
\label{R-P}
|\Lambda_{\sym}(j)|\leq(n+1)\Lambda(j).
\end{equation}
where $\Lambda(j)$ is the classical von Mangoldt function.  Furthermore, if $\gcd(j,N)=1$, then
\begin{equation}
\label{von-Mangoldt-def}
\Lambda_{\sym}(j)=
\begin{cases}
U_n(\cos(m\theta_p))\log(p)  &\mbox{if $j=p^m$, $m>0$,} \\
0 & \mbox{otherwise,}
\end{cases} 
\end{equation}
where $U_n(x)$ is the $n$-th Chebyshev polynomial of the second type.
%Furthermore, it follows from basic properties of Chebyshev polynomials of the second type that if $n\geq2$ and $p\nmid N$, then
%\begin{equation}
%\label{cos-to-von-Mangoldt}
%\Lambda_{\sym}(p)-\Lambda_{\text{Sym}^{n-2} f}(p)=2\cos(n\theta_p)\log(p).
%\end{equation}

%%%%%%%%%%%%%%%%%%%%%%%%%%%%%%%%%%%%%%%%%

\subsection{The analytic conductor}

We want to estimate the {\it analytic conductor}
\begin{equation}
\mathfrak{q}_{\sym}(s)=q_{\sym} \prod_{j=1}^{n+1}(|s+\kappa_{j,\sym}|+3).
\end{equation}
An estimate of the analytic conductor will allow us to easily make estimates for $\Lsym$ that are uniform as we change the symmetric power $n$, the weight $k$, and the level $N$ of $f$.  Most importantly, we want an estimate of $\mathfrak{q}_{\sym}$ as $n\to\infty$.  The quality of the error term in the Sato-Tate Conjecture depends on how well one can estimate $\mathfrak{q}_{\sym}$ as a function of $n$.  We begin with an estimate for $q_{\sym}$ given by Lemma 2.1 of \cite{Rouse}.

%%%%%%%%%%%%%%%%%%%%%%%%%%%%%%%%%%%%%%%%%

\begin{lemma}
\label{conductor}
As $n\to\infty$, we have $\log(q_{\sym})\ll_{N} n^3$.
\end{lemma}

\begin{remark}
Under our assumptions of automorphy and functoriality, Cogdell and Michel prove \cite{CM} that if $N$ is squarefree, then $\log(q_{\sym})=n\log(N)$.  With this improvement, the assumption of a squarefree level $N$ provides considerable improvement over Lemma \ref{conductor} when GRH is assumed.  However, it will not provide any improvement without GRH because of the specific dependence of our zero-free region for $\Lsym$ on $n$.
\end{remark}

%%%%%%%%%%%%%%%%%%%%%%%%%%%%%%%%%%%%%%%%%%%

From Lemma \ref{conductor} and the shape of the numbers $\kappa_{j,\sym}$, we may conclude the following.

\begin{lemma}
\label{analytic conductor}
As $n\to\infty$, we have
\[
\log(\mathfrak{q}_{\sym}(0))\ll_{k,N} n^3,\quad\log(\mathfrak{q}_{\sym}(iT))\ll_{k,N} n^3+n\log(T).
\]
\end{lemma}

Lemma \ref{analytic conductor} also allows us to determine the distribution of nontrivial zeros in the critical strip by measuring the quantity
\[
N(T,\sym)=\#\{\rho=\beta+i\gamma:0\leq\beta\leq1,|\gamma|\leq T,L(\sym,\rho)=0\}.
\]
By Theorem 5.8 of \cite{IK}, we have
\[
N(T,\sym)=\frac{T}{\pi}\log\left(\frac{q_{\sym}T^{n+1}}{(2\pi e)^{n+1}}\right)+O(\log(\mathfrak{q}_{\sym}(iT))).
\]
Using Lemma \ref{analytic conductor} to give us a complete description of the dependence of $N(T,\sym)$ on $n$, we obtain the following result, which is part of the proof of Lemma 3.4 in \cite{Rouse}.

\begin{lemma}
\label{zero-count}
As $T\to\infty$, we have
\[
\mathcal{N}(T,\sym):=N(T+1,\sym)-N(T,\sym)\ll_{k,N} n^3+n\log(T).
\]
\end{lemma}

%%%%%%%%%%%%%%%%%%%%%%%%%%%%%%%%%%%%%%%%%%%

\section{Preliminary Setup}

If $\chi_I$ is the indicator function of the interval $I=[\alpha,\beta]$, then we have
\begin{equation}
\label{estimate}
\pi_{f,I}(x)=\sum_{p\leq x}\chi_{I}(\theta_p).
\end{equation}
We approximate $\chi_{I}$ with a differentiable function using the following construction.

%%%%%%%%%%%%%%%%%%%%%%%%%%%%%%%%%%%%%%%%%%%

\begin{lemma}[Lemma 12 of \cite{vino}]
\label{Vinogradov}
Let $R$ be a positive integer, and let $a,b,\delta\in\R$ satisfy
\[
0<\delta<1/2,\quad\delta\leq b-a\leq1-\delta.
\]
Then there exists an even periodic function $g(y)$ with period 1 satisfying
\begin{enumerate}
\item $g(y)=1$ when $y\in[a+\frac{1}{2}\delta,b-\frac{1}{2}\delta]$,
\item $g(y)=0$ when $y\in[b+\frac{1}{2}\delta,1+a-\frac{1}{2}\delta]$,
\item $0\leq g(y)\leq1$ when $y$ is in the rest of the interval $[a-\frac{1}{2}\delta,1+a-\frac{1}{2}\delta]$, and
\item $g(y)$ has the Fourier expansion
\[
g(y)=b-a+\add{n=1}{\infty}(a_n\cos(2\pi nx)+b_m\sin(2\pi nx)),
\]
where for all $n\geq 1$,
\[
|a_n|,|b_n|\leq\min\left\{2(b-a),\frac{2}{n\pi},\frac{2}{n\pi}\p{\frac{R}{\pi n\delta}}^R\right\}.
\]
\end{enumerate}
\end{lemma}

%%%%%%%%%%%%%%%%%%%%%%%%%%%%%%%%%%%%%%%%%%%

Let $g(\theta)$ be defined as in Lemma \ref{Vinogradov}, where $a=\frac{\alpha}{2\pi}-\frac{\delta}{2}$, and $b=\frac{\beta}{2\pi}+\frac{\delta}{2}$.  We will choose $\delta$ to be a function of $x$ that tends to zero as $x$ tends to infinity, and we will choose $R$ to ensure the absolute convergence of the Fourier series.  Define $g^+(\theta;I,\delta)=g(\frac{\theta}{2\pi})+g(-\frac{\theta}{2\pi})$, which equals 1 for $\theta\in I$, equals 0 for $\theta\in[0,\alpha-2\pi\delta]\cup[\beta+2\pi\delta,\pi]$, and is between 0 and 1 elsewhere in the interval $[0,\pi]$.  Thus $g^+(\theta;I,\delta)$ a pointwise upper bound for $\chi_{I}(\theta)$.  By repeating this construction with $a=\frac{\alpha}{2\pi}+\frac{\delta}{2}$, and $b=\frac{\beta}{2\pi}-\frac{\delta}{2}$, we can obtain a lower bound for $\chi_I(\theta)$, say $g^-(\theta;I,\delta)$.  To ensure that $g^-(\theta;I,\delta)$ is in fact a lower bound for $\chi_I(\theta)$, we require that $\beta-\alpha>2\pi\delta$, which is ensured when $x$ is sufficiently large because $I$ is fixed.

%%%%%%%%%%%%%%%%%%%%%%%%%%%%%%%%%%%%%%%%%%%

We can express $g^{\pm}(\theta;I,\delta)$ with respect to the basis of Chebyshev polynomials of the second kind $\{U_n(\cos(\theta))\}_{n=0}^\infty$, which is an orthonormal basis for $L^2([0,\pi],\mu_{ST})$.  Specifically,
\begin{equation}
\label{approx}
g^{\pm}(\theta;I,\delta)=a_0^{\pm}(I,\delta)-a_2^{\pm}(I,\delta)+\add{n=1}{\infty}(a_n^{\pm}(I,\delta)-a_{n+2}^{\pm}(I,\delta))U_n(\cos(\theta)),
\end{equation}
where $a_n^{\pm}(I,\delta)$ is the $n$-th Fourier coefficient in the cosine expansion of $g^{\pm}(\theta;I,\delta)$.  From Lemma \ref{Vinogradov}, we have
\begin{align}
\label{fourier-bounds}
|a_0(I,\pm\delta)-a_2(I,\pm\delta)-\mu_{ST}(I)|&\ll \delta,\\
|a_n(I,\pm\delta)-a_{n+2}(I,\pm\delta)|&\leq\frac{4}{n\pi}\p{\frac{R}{\pi n\delta}}^R \text{ for $n\geq1$}.\notag
\end{align}

When summing $g^{\pm}(\theta_p;I,\delta)$ over primes $p\leq x$, we may switch the order of summation because we choose $R$ to ensure absolute convergence.   Using (\ref{estimate}), (\ref{approx}), (\ref{fourier-bounds}), and the prime number theorem, we have that if
\begin{equation}
\label{Phi}
\Phi_{\sym}(x)=\sum_{p\leq x}U_n(\cos(\theta_p)),
\end{equation}
then
\begin{equation}
\label{grand-estimate}
\pi_{f,I}(x)=\mu_{ST}(I)\pi(x)+O\left(\frac{\delta x}{\log(x)}+\sum_{n=1}^\infty \frac{1}{n}\left(\frac{R}{n\delta}\right)^R |\Phi_{\sym}(x)|\right).
\end{equation}

Theorem \ref{main-theorem} will now follow from an estimate of $\Phi_{\sym}(x)$ and choosing $\delta$ and $R$ optimally.  The goal of the next section is to estimate $\Phi_{\sym}(x)$, which proceeds very much like the classical prime number theorem.

%%%%%%%%%%%%%%%%%%%%%%%%%%%%%%%%%%%%%%%%

\section{Estimating $\Phi_{\sym}(x)$}

%%%%%%%%%%%%%%%%%%%%%%%%%%%%%%%%%%%%%%%%

By our assumption of functoriality, $|\alpha_p|,|\beta_p|\leq1$ for all primes $p$, and $|\alpha_p|,|\beta_p|=1$ for all $p\nmid N$.  Thus $\Lsym$ satisfies the Ramanujan-Petersson Conjecture for all $n\geq1$.  As such, we may use Equation 5.53 from Chapter 5 of \cite{IK} to estimate the summatory von-Mangoldt function for $\Lsym$ given by
\[
\psi_{\sym}(x)=\sum_{j\leq x}\Lambda_{\sym}(j).
\]
\begin{lemma}
\label{explicit-formula}
If $n\geq1$, then
\[
\psi_{\sym}(x)=-\sum_{|\gamma|\leq T}\frac{x^\rho}{\rho}+O\left(\frac{x}{T}\log(x)\log(x^{n+1}\mathfrak{q}_{\sym}(0))\right),
\]
where $\rho=\beta+i\gamma$ runs over the zeros of $\Lsym$ in the critical strip of height up to $T$, with any $1\leq T\leq x$, and the implied constant is absolute.
\end{lemma}

By our assumption of Langlands functoriality, there will be no exceptional real zeros close to $s=1$ (see Section 4 of \cite{HR}).  As such, it remains to estimate the sum over nontrivial zeros.  We invoke the zero-free region given in Theorem 5.10 of \cite{IK}, which is currently the best zero-free region for a generic automorphic $L$-function.

\begin{lemma}
\label{zero-free}
There exists an absolute constant $c>0$ such that $\Lsym$ has no zeros in the region
\[
s=\sigma+it,\qquad \sigma\geq1-\frac{c}{(n+1)^4\log(\mathfrak{q}_{\sym}(0)(|t|+3))}.
\]
\end{lemma}

Using Lemmata \ref{explicit-formula} and \ref{zero-free}, we estimate $|\Phi_{\sym}(x)|$.

\begin{lemma}
\label{PNT}
Assume the above notation, and let $n\geq1$.  For some constant $0<c_2<c$ (depending only on $N$ and $k$), we have
\[
|\Phi_{\sym}(x)|\ll_{k,N} n^{3}x\exp\left(-\frac{c_{2}\log(x)}{n^4(\sqrt{\log(x)}+n^3)}\right).
\]
\end{lemma}

\begin{proof}
We begin by estimating the sum over zeros in Lemma \ref{explicit-formula}.  By Lemma \ref{zero-free}, if $\rho=\beta+i\gamma$ is a nontrivial zero of $\Lsym$, then
\[
|x^\rho|\leq|x^\beta|\leq x\exp\left(-\frac{c\log(x)}{(n+1)^4\log(\mathfrak{q}_{\sym}(0)(|t|+3))}\right).
\]
Thus
\[
\sum_{|\gamma|\leq T}\left|\frac{x^\rho}{\rho}\right|\ll x\exp\left(-\frac{c\log(x)}{(n+1)^4\log(\mathfrak{q}_{\sym}(0)(|T|+3))}\right)\sum_{j\leq T}\frac{\mathcal{N}(j,\sym)}{j}.
\]
Now, Lemma \ref{zero-count} tells us that the sum over zeros is
\[
\ll_{k,N} n^3(\log T)^2x\exp\left(-\frac{c\log(x)}{(n+1)^4\log(\mathfrak{q}_{\sym}(0)(|T|+3))}\right)
\]

To address the error term in Lemma \ref{explicit-formula}, we use Lemma \ref{analytic conductor} to obtain
\[
\frac{x}{T}\log(x)\log(x^{n+1}\mathfrak{q}_{\sym}(0))\ll_{k,N} \frac{n^3x}{T}(\log x)^2.
\]
To balance our estimate for the sum over nontrivial zeros with the error term in Lemma \ref{explicit-formula}, we choose $T=\exp(\sqrt{\log x})$ to obtain
\begin{equation}
\label{PNT-version-1}
\psi_{\sym}(x)\ll_{k,N} n^3 x\exp\left(-\frac{c_1\log(x)}{n^4(\sqrt{\log (x)}+n^3)}\right)
\end{equation}
for some $0<c_1<c$.  By a standard application of Abel summation, one has
\[
\Psi_{\sym}(x):=\sum_{j\leq x}\frac{\Lambda_{\sym}(j)}{\log(j)}=\frac{\psi_{\sym}(x)}{\log(x)}+\int_2^x\frac{\psi_{\sym}(t)}{t(\log t)^2}~dt.
\]
Applying (\ref{PNT-version-1}), we have that for some constant $0<c_2<c_1$ (depending only on $N$ and $k$),
\begin{equation}
\Psi_{\sym}(x)\ll_{k,N} n^{3}x\exp\left(-\frac{c_{2}\log(x)}{n^4(\sqrt{\log(x)}+n^3)}\right).
\end{equation}

We now show that $|\Phi_{\sym}(x)-\Psi_{\sym}(x)|$ is small.  By (\ref{von-Mangoldt-def}), if $p\nmid N$ is prime, then
\[
\frac{\Lambda_{\sym}(p)}{\log(p)}=U_n(\cos(\theta_p)).
\]
At all other prime powers $j=p^m$, it follows from (\ref{R-P}) that
\[
\left|\frac{\Lambda_{\sym}(j)}{\log(j)}\right|\leq n+1.
\]
Finally, we have $|U_n(\cos(\theta_p))|\leq n+1$ for all $p$ by basic properties of Chebyshev polynomials.  Therefore,
\[
|\Phi_{\sym}(x)-\Psi_{\sym}(x)|\leq(n+1)\left(\sum_{\substack{m\geq2 \\ p^m\leq x}}1+\sum_{p\mid N}1\right)\ll_{N}n\sqrt{x}.
\]
This error is negligible, so we have proven the desired result.
\end{proof}

%%%%%%%%%%%%%%%%%%%%%%%%%%%%%%%%%%%%%%%%%%%

\section{Proof of Theorem \ref{main-theorem}}

%%%%%%%%%%%%%%%%%%%%%%%%%%%%%%%%%%%%%%%%%%%

To prove Theorem \ref{main-theorem}, it remains to choose $\delta$ and $R$ so that the error term in (\ref{grand-estimate}) is minimized.  The factor of $n^{3}$ in Lemma \ref{PNT} tells us that we must take $R$ to be at least 4 in (\ref{grand-estimate}) to ensure absolute convergence of the sum in the error term.  It follows from Lemma \ref{PNT} that
\begin{align*}
\sum_{n=1}^\infty\frac{1}{n}\left(\frac{R}{n\delta}\right)^R|\Phi_{\sym}(x)|\ll_{k,N}\;& \sum_{n=1}^\infty\frac{1}{n}\left(\frac{R}{n\delta}\right)^R n^{3}x\exp\left(-\frac{c_{2}\log(x)}{n^4(\sqrt{\log(x)}+n^3)}\right)\\
\ll_{k,N}\;& \frac{R^R}{\delta^R}x\int_{1}^\infty\frac{1}{t^{R-2}}\exp\left(-\frac{c_{2}\sqrt{\log(x)}}{t^4}\right)~dt\\
\ll_{k,N}\;& \frac{R^{\frac{5R}{4}}}{\delta^R}\frac{x}{(\log x)^{\frac{R-3}{8}}}.
\end{align*}
We have thus reduced (\ref{grand-estimate}) to
\begin{equation}
\label{grand-est-reduced}
\pi_{f,I}(x)=\mu_{ST}(I)\pi(x)+O_{k,N}\left(\frac{\delta x}{\log(x)}+\frac{R^{\frac{5R}{4}}}{\delta^R}\frac{x}{(\log x)^{\frac{R-3}{8}}}\right).
\end{equation}
Choosing
\[
\delta=R^{\frac{5}{4}}\log(x)^{\frac{3}{2R}-\frac{1}{8}},
\]
we balance the error term in (\ref{grand-est-reduced}), which is now of order
\[
\ll_{k,N} \frac{\delta x}{\log(x)}\ll_{k,N} R^{\frac{5}{4}}\frac{x}{(\log x)^{\frac{9}{8}-\frac{3}{2R}}}.
\]
Since we can choose $R$ to be a finite, arbitrarily large integer, we obtain the bound claimed in Theorem \ref{main-theorem}.

\subsection*{Acknowledgements}

%%%%%%%%%%%%%%%%%%%%%%%%%%%%%%%%%%%%%%%%%%%

The author thanks David Borthwick, Ken Ono, and Jeremy Rouse for their comments and support.  The author also thanks the anonymous referee for additional comments.

%%%%%%%%%%%%%%%%%%%%%%%%%%%%%%%%%%%%%%%%%%%

\bibliography{STPaper}

\def\cprime{$'$}
\begin{bibdiv}
\begin{biblist}

\bib{sato-tate}{article}{
      author={Barnet-Lamb, Tom},
      author={Geraghty, David},
      author={Harris, Michael},
      author={Taylor, Richard},
       title={A family of {C}alabi-{Y}au varieties and potential automorphy
  {II}},
        date={2011},
        ISSN={0034-5318},
     journal={Publ. Res. Inst. Math. Sci.},
      volume={47},
      number={1},
       pages={29\ndash 98},
         url={http://dx.doi.org/10.2977/PRIMS/31},
      review={\MR{2827723 (2012m:11069)}},
}

\bib{BK}{article}{
      author={Bucur, A.},
      author={Kedlaya, K.},
       title={An {A}pplication of the {E}ffective {S}ato-{T}ate {C}onjecture},
         url={http://arxiv.org/abs/1301.0139},
        note={Preprint},
}

\bib{CM}{article}{
      author={Cogdell, J.},
      author={Michel, P.},
       title={On the complex moments of symmetric power {$L$}-functions at
  {$s=1$}},
        date={2004},
        ISSN={1073-7928},
     journal={Int. Math. Res. Not.},
      number={31},
       pages={1561\ndash 1617},
      review={\MR{MR2035301 (2005f:11094)}},
}

\bib{GJ}{article}{
      author={Gelbart, Stephen},
      author={Jacquet, Herv{\'e}},
       title={A relation between automorphic representations of {${\rm GL}(2)$}
  and {${\rm GL}(3)$}},
        date={1978},
        ISSN={0012-9593},
     journal={Ann. Sci. \'Ecole Norm. Sup. (4)},
      volume={11},
      number={4},
       pages={471\ndash 542},
         url={http://www.numdam.org/item?id=ASENS_1978_4_11_4_471_0},
      review={\MR{533066 (81e:10025)}},
}

\bib{HR}{article}{
      author={Hoffstein, Jeffrey},
      author={Ramakrishnan, Dinakar},
       title={Siegel zeros and cusp forms},
        date={1995},
        ISSN={1073-7928},
     journal={Internat. Math. Res. Notices},
      number={6},
       pages={279\ndash 308},
         url={http://dx.doi.org/10.1155/S1073792895000225},
      review={\MR{1344349 (96h:11040)}},
}

\bib{IK}{book}{
      author={Iwaniec, H.},
      author={Kowalski, E.},
       title={Analytic number theory},
      series={American Mathematical Society Colloquium Publications},
   publisher={American Mathematical Society},
     address={Providence, RI},
        date={2004},
      volume={53},
        ISBN={0-8218-3633-1},
      review={\MR{MR2061214 (2005h:11005)}},
}

\bib{Kim}{article}{
      author={Kim, Henry~H.},
       title={Functoriality for the exterior square of {${\rm GL}_4$} and the
  symmetric fourth of {${\rm GL}_2$}},
        date={2003},
        ISSN={0894-0347},
     journal={J. Amer. Math. Soc.},
      volume={16},
      number={1},
       pages={139\ndash 183},
         url={http://dx.doi.org/10.1090/S0894-0347-02-00410-1},
        note={With appendix 1 by Dinakar Ramakrishnan and appendix 2 by Kim and
  Peter Sarnak},
      review={\MR{1937203 (2003k:11083)}},
}

\bib{KS1}{article}{
      author={Kim, Henry~H.},
      author={Shahidi, Freydoon},
       title={Cuspidality of symmetric powers with applications},
        date={2002},
        ISSN={0012-7094},
     journal={Duke Math. J.},
      volume={112},
      number={1},
       pages={177\ndash 197},
         url={http://dx.doi.org/10.1215/S0012-9074-02-11215-0},
      review={\MR{1890650 (2003a:11057)}},
}

\bib{KS2}{article}{
      author={Kim, Henry~H.},
      author={Shahidi, Freydoon},
       title={Functorial products for {${\rm GL}_2\times{\rm GL}_3$} and the
  symmetric cube for {${\rm GL}_2$}},
        date={2002},
        ISSN={0003-486X},
     journal={Ann. of Math. (2)},
      volume={155},
      number={3},
       pages={837\ndash 893},
         url={http://dx.doi.org/10.2307/3062134},
        note={With an appendix by Colin J. Bushnell and Guy Henniart},
      review={\MR{1923967 (2003m:11075)}},
}

\bib{MS}{article}{
      author={Moreno, C.~J.},
      author={Shahidi, F.},
       title={The {$L$}-functions {$L(s,{\rm Sym}^m(r),\pi)$}},
        date={1985},
        ISSN={0008-4395},
     journal={Canad. Math. Bull.},
      volume={28},
      number={4},
       pages={405\ndash 410},
         url={http://dx.doi.org/10.4153/CMB-1985-049-4},
      review={\MR{812115 (87a:11051)}},
}

\bib{Murty}{article}{
      author={Murty, V.~Kumar},
       title={Explicit formulae and the {L}ang-{T}rotter conjecture},
        date={1985},
        ISSN={0035-7596},
     journal={Rocky Mountain J. Math.},
      volume={15},
      number={2},
       pages={535\ndash 551},
         url={http://dx.doi.org/10.1216/RMJ-1985-15-2-535},
        note={Number theory (Winnipeg, Man., 1983)},
      review={\MR{823264 (87h:11051)}},
}

\bib{RT}{article}{
      author={Rouse, J.},
      author={Thorner, J.},
       title={The explicit {S}ato-{T}ate conjecture and densities pertaining to
  {L}ehmer-type questions},
     journal={preprint},
}

\bib{Rouse}{article}{
      author={Rouse, Jeremy},
       title={Atkin-{S}erre type conjectures for automorphic representations on
  {${\rm GL}(2)$}},
        date={2007},
        ISSN={1073-2780},
     journal={Math. Res. Lett.},
      volume={14},
      number={2},
       pages={189\ndash 204},
      review={\MR{2318618 (2009e:11080)}},
}

\bib{Shahidi}{incollection}{
      author={Shahidi, Freydoon},
       title={Symmetric power {$L$}-functions for {${\rm GL}(2)$}},
        date={1994},
   booktitle={Elliptic curves and related topics},
      series={CRM Proc. Lecture Notes},
      volume={4},
   publisher={Amer. Math. Soc., Providence, RI},
       pages={159\ndash 182},
      review={\MR{1260961 (95c:11066)}},
}

\bib{vino}{book}{
      author={Vinogradov, I.~M.},
       title={The method of trigonometrical sums in the theory of numbers},
   publisher={Dover Publications Inc.},
     address={Mineola, NY},
        date={2004},
        ISBN={0-486-43878-3},
        note={Translated from the Russian, revised and annotated by K. F. Roth
  and Anne Davenport, Reprint of the 1954 translation},
      review={\MR{2104806 (2005f:11172)}},
}

\end{biblist}
\end{bibdiv}
\end{document}